\documentclass{amsart}
\usepackage{amsmath}
\usepackage{amsfonts}

\newtheorem{theorem}{Theorem}
\theoremstyle{plain}

\newtheorem{corollary}{Corollary}

\newtheorem{definition}{Definition}

\newtheorem{proposition}{Proposition}

\numberwithin{equation}{section}

\begin{document}
\title[Magnetic Curves in $C$-manifolds]{Magnetic Curves in $C$-manifolds}
\author{\c{S}aban G\"{u}ven\c{c}}
\address[\c{S}. G\"{u}ven\c{c}]{Balikesir University, Department of
Mathematics\\
10145, \c{C}a\u{g}\i \c{s}, Bal\i kesir, TURKEY}
\email[\c{S}. G\"{u}ven\c{c}]{sguvenc@balikesir.edu.tr}

\begin{abstract}
In this paper, we study normal magnetic curves in $C$-manifolds. We prove
that magnetic trajectories with respect to the contact magnetic fields are
indeed $\theta _{\alpha }$-slant curves with certain curvature functions.
Then, we give the parametrizations of normal magnetic curves in $\mathbb{R}^{2n+s}$ with its structures as a $C$-manifold.
\end{abstract}

\maketitle

\section{Introduction}

Let $(M,g)$ be a Riemannian manifold, $F$ a closed $2$-form and let us
denote the Lorentz force on $M$ by $\Phi $, which is a $(1,1)$-type tensor
field.\ If $F$ is associated by the relation%
\begin{equation}
g(\Phi X,Y)=F(X,Y),\text{ \ }\forall X,Y\in \chi (M),  \label{F}
\end{equation}%
then it is called a\textit{\ magnetic field} (\cite{Adachi-1996}, \ \cite%
{BRCF} and \cite{Comtet-1987}). Let $\nabla $ be the Riemannian connection
associated to the Riemannian metric $g$ and $\gamma :I\rightarrow M$ a
smooth curve. If $\gamma $ satisfies the Lorentz equation%
\begin{equation}
\nabla _{\gamma ^{\prime }(t)}\gamma ^{\prime }(t)=\Phi (\gamma ^{\prime
}(t)),  \label{Lorentz eq}
\end{equation}%
then it is called a \textit{magnetic curve} or a \textit{trajectory} for the
magnetic field $F$. The Lorentz equation can be considered as a
generalization of the equation for geodesics. Magnetic trajectories have
constant speed. If the speed of the magnetic curve $\gamma $ is equal to $1$%
, then it is called a \textit{normal magnetic curve }\cite{DIMN-2015}. For
fundamentals of almost contact metric manifolds, we refer to Blair's book 
\cite{Blair-2010}. This paper is based on a similar idea of Ozgur and the
present author's previous paper \cite{GO-2019}.

\section{Preliminaries}

Let $\left( M^{2n+s},g\right) $ be a differentiable manifold, $\varphi $ a $%
(1,1)$-type tensor field, $\eta ^{\alpha }$ 1-forms, $\xi _{\alpha }$ vector
fields for $\alpha =1,2,...,s$, satisfying%
\begin{equation}
\varphi ^{2}X=-X+\sum\limits_{\alpha =1}^{s}\eta ^{\alpha }\left( X\right)
\xi _{\alpha },  \label{2.1}
\end{equation}

\begin{equation*}
\eta ^{\alpha }\left( \xi _{\beta }\right) =\delta _{\beta }^{\alpha },\text{
}\varphi \xi _{\alpha }=0,\text{ }\eta ^{\alpha }\left( \varphi X\right) =0,%
\text{ }\eta ^{\alpha }\left( X\right) =g\left( X,\xi _{\alpha }\right) ,
\end{equation*}%
\begin{equation}
g(\varphi X,\varphi Y)=g(X,Y)-\overset{s}{\underset{\alpha =1}{\sum }}\eta
^{\alpha }(X)\eta ^{\alpha }(Y),  \label{eq2}
\end{equation}%
where $X,Y\in TM$. Then $(\varphi ,\xi _{\alpha },\eta ^{\alpha },g)$ is
called \textit{framed }$\varphi $\textit{-structure} and $(M^{2n+s},\varphi
,\xi _{\alpha },\eta ^{\alpha },g)$ is called \textit{framed }$\varphi $%
\textit{-manifold. }The\textit{\ fundamental 2-form} and\textit{\ Nijenhuis
tensor }is given by:\textit{\ }%
\begin{equation*}
\Omega (X,Y)=g\left( X,\varphi Y\right) ,
\end{equation*}%
\begin{equation*}
N_{\varphi }\left( X,Y\right) =-2\sum\limits_{\alpha =1}^{s}d\eta ^{\alpha
}\left( X,Y\right) \xi _{\alpha }.
\end{equation*}%
If $d\Omega =0$ and $d\eta ^{\alpha }=0$, $M=(M,\varphi ,\xi _{\alpha },\eta
^{\alpha },g)$ is called a $C$-manifold. In a $C$-manifold, it is known that%
\begin{equation*}
\left( \nabla _{X}\varphi \right) Y=0
\end{equation*}%
and%
\begin{equation*}
\nabla _{X}\xi _{\alpha }=0,
\end{equation*}%
(see \cite{Blair-1970} and \cite{Blair-2010}).

\section{Magnetic Curves in $C$-manifolds}

Let $\gamma :I\rightarrow M$ \ be a unit-speed curve\ in an $n$-dimensional
Riemannian manifold $(M,g)$. The curve $\gamma $ is called a\textit{\ Frenet
curve of osculating order }$r$ $\left( 1\leq r\leq n\right) $, if there
exists orthonormal vector fields $T,v_{2},...,v_{r}$ along the curve
validating the Frenet equations 
\begin{eqnarray}
T &=&\gamma ^{\prime },  \notag \\
\nabla _{T}T &=&\kappa _{1}v_{2},  \notag \\
\nabla _{T}v_{2} &=&-\kappa _{1}v+\kappa _{2}v_{3},  \label{3.1} \\
&&...  \notag \\
\nabla _{T}v_{r} &=&-\kappa _{r-1}v_{r-1},  \notag
\end{eqnarray}%
where $\kappa _{1},...,\kappa _{r-1}$ are positive functions called the
curvatures of $\gamma $. If $\kappa _{1}=0,$ then $\gamma $ is called a 
\textit{geodesic}. If $\kappa _{1}$ is a non-zero positive constant and $%
r=2, $ $\gamma $ is called a \textit{circle.} If $\kappa _{1},...,\kappa
_{r-1}$ are non-zero positive constants, then $\gamma $ is called a \textit{%
helix of order }$r$ $\left( r\geq 3\right) .$ If $r=3$, it is shortly called
a \textit{helix}.\bigskip

A submanifold of an $C$-manifold is said to be an \textit{integral
submanifold} if $\eta ^{\alpha }(X)=0,$ $\alpha \in \left\{
1,2,...,s\right\} ,$ where $X$ is tangent to the submanifold. \ A \textit{%
Legendre curve} is a $1$-dimensional integral submanifold of an $C$-manifold 
$(M^{2n+s},\varphi ,\xi _{\alpha },\eta ^{\alpha },g)$\textit{.} More
precisely, a unit-speed curve $\gamma :I\rightarrow M$ is a Legendre curve%
\textit{\ }if $T$ is g-orthogonal to all $\xi _{\alpha }$ $\left( \alpha
=1,2,...s\right) $, where $T=\gamma ^{\prime }.$

\begin{definition}
Let $\gamma $ be a unit-speed curve in a $C$-manifold $(M,\varphi ,\xi
_{\alpha },\eta ^{\alpha },g)$. $\gamma $ is called a $\theta _{\alpha }-$%
slant curve if there exist constant contact angles such that $\eta ^{\alpha
}(T)=\cos \theta _{\alpha }$, $\alpha =1,2,...,s$. If $\theta _{\alpha
}=\theta $ for all $\alpha =1,2,...,s$, then $\gamma $ is shortly called
slant. Moreover, if $\theta _{\alpha }=\frac{\pi }{2}$ for all $\alpha
=1,2,...,s$, then $\gamma $ is called a Legendre curve.
\end{definition}

For $\theta _{\alpha }-$slant curves, we can give the following inequality
for the constant contact angles:%
\begin{equation*}
\sum\limits_{\alpha =1}^{s}\cos ^{2}\theta _{\alpha }\leq 1.
\end{equation*}%
The equality case is only valid when $\gamma $ is a geodesic as an integral
curve of $\pm \sum\limits_{\alpha =1}^{s}\cos \theta _{\alpha }\xi _{\alpha
} $.

Let $\gamma $ be a unit-speed Legendre curve in a $C$-manifold $(M,\varphi
,\xi _{\alpha },\eta ^{\alpha },g)$. If we differentiate $\eta ^{\alpha
}(T)=0$, we obtain $\eta ^{\alpha }(v_{2})=0$. We can continue this process
until we find $\eta ^{\alpha }(v_{r})=0$. Thus, we can state the following
proposition:

\begin{proposition}
If $\gamma $ is a unit-speed Legendre curve in a $C$-manifold $(M,\varphi
,\xi _{\alpha },\eta ^{\alpha },g)$, then $\xi _{\alpha }$ is $g$-orthogonal
to $sp\left\{ T,v_{2},...,v_{r}\right\} $, for all $\alpha =1,2,...,s$.
\end{proposition}

If we consider equations (\ref{F}), (\ref{Lorentz eq}) and (\ref{3.1})
together, for a normal magnetic curve of a magnetic field $F$ with charge $q$%
, we find%
\begin{equation*}
\nabla _{T}T=\Phi T,
\end{equation*}%
\begin{equation*}
F\left( X,Y\right) =g\left( \Phi X,Y\right) ,
\end{equation*}%
\begin{eqnarray*}
F_{q}\left( X,Y\right) &=&q\Omega \left( X,Y\right) \\
&=&qg\left( X,\varphi Y\right) ,
\end{eqnarray*}%
which gives us

\begin{equation*}
\Phi _{q}=-q\varphi .
\end{equation*}%
Here, $T$ denotes the tangential vector field of the normal magnetic curve $%
\gamma $ for the magnetic field $F_{q}$ in $M$. Then, we have the following
equations:%
\begin{equation}
\nabla _{T}T=-q\varphi T,  \label{2.6}
\end{equation}%
\begin{equation*}
\nabla _{T}\xi _{\alpha }=0,
\end{equation*}%
\begin{eqnarray*}
\nabla _{T}\varphi T &=&\left( \nabla _{T}\varphi \right) T+\varphi \nabla
_{T}T \\
&=&\varphi \left( -q\varphi T\right) \\
&=&-q\varphi ^{2}T \\
&=&-q\left( -T+\sum\limits_{\alpha =1}^{s}\eta ^{\alpha }\left( T\right) \xi
_{\alpha }\right) \\
&=&qT-q\sum\limits_{\alpha =1}^{s}\eta ^{\alpha }\left( T\right) \xi
_{\alpha }.
\end{eqnarray*}%
If we take the inner product of equation (\ref{2.6}) with $\xi _{\alpha }$,
we obtain 
\begin{eqnarray*}
0 &=&g\left( -q\varphi T,\xi _{\alpha }\right) =g\left( \nabla _{T}T,\xi
_{\alpha }\right) \\
&=&\frac{d}{dt}g\left( T,\xi _{\alpha }\right) .
\end{eqnarray*}%
Integrating both sides, we get%
\begin{equation*}
\eta ^{\alpha }(T)=\cos \theta _{\alpha }=constant,
\end{equation*}%
for all $\alpha =1,2,...,s$. Equations (\ref{3.1}) and (\ref{2.6}) give us%
\begin{equation}
\nabla _{T}T=\kappa _{1}v_{2}=-q\varphi T,  \label{star}
\end{equation}%
\begin{eqnarray*}
g\left( \varphi T,\varphi T\right) &=&g\left( T,T\right)
-\sum\limits_{\alpha =1}^{s}\left( \eta ^{\alpha }\left( T\right) \right)
^{2} \\
&=&1-\sum\limits_{\alpha =1}^{s}\cos ^{2}\theta _{\alpha }
\end{eqnarray*}%
and%
\begin{equation*}
\left\Vert \varphi T\right\Vert =\sqrt{1-\sum\limits_{\alpha =1}^{s}\cos
^{2}\theta _{\alpha }}.
\end{equation*}%
From equation (\ref{star}), we find%
\begin{equation}
\kappa _{1}=\left\vert q\right\vert \sqrt{1-\sum\limits_{\alpha =1}^{s}\cos
^{2}\theta _{\alpha }}=constant,  \label{3.6}
\end{equation}%
\begin{equation*}
-q\varphi T=\kappa _{1}v_{2}=\left\vert q\right\vert \sqrt{%
1-\sum\limits_{\alpha =1}^{s}\cos ^{2}\theta _{\alpha }}v_{2}
\end{equation*}%
and%
\begin{equation}
\varphi T=-sgn(q)\sqrt{1-\sum\limits_{\alpha =1}^{s}\cos ^{2}\theta _{\alpha
}}v_{2}.  \label{3.7}
\end{equation}%
If $\kappa _{2}=0,$ then $r=2$ and $\gamma $ is a circle. If we apply $\eta
^{\alpha }$ to equation (\ref{3.7}), we obtain%
\begin{equation*}
\eta ^{\alpha }\left( v_{2}\right) =0,
\end{equation*}%
which gives us%
\begin{eqnarray*}
\nabla _{T}\eta ^{\alpha }\left( v_{2}\right) &=&0 \\
&=&g\left( \nabla _{T}v_{2},\xi _{\alpha }\right) +g\left( T,\nabla _{T}\xi
_{\alpha }\right) \\
&=&-\kappa _{1}\cos \theta _{\alpha }.
\end{eqnarray*}%
As a result, we get $\cos \theta _{\alpha }=0,$ for all $\alpha =1,2,...,s$.
Hence, $\gamma $ is a Legendre circle, $\left\Vert \varphi T\right\Vert =1$
and $\kappa _{1}=\left\vert q\right\vert $. Let $\kappa _{2}\neq 0$. Using
equations (\ref{2.1}) and (\ref{3.1}), we calculate%
\begin{eqnarray}
\nabla _{T}\varphi T &=&\left( \nabla _{T}\varphi \right) T+\varphi \nabla
_{T}T  \notag \\
&=&\varphi \left( -q\varphi T\right)  \label{3.10} \\
&=&-q\left( -T+\sum\limits_{\alpha =1}^{s}\cos \theta _{\alpha }\xi _{\alpha
}\right) .  \notag
\end{eqnarray}%
Differentiating equation (\ref{3.7}), we also have%
\begin{equation}
\nabla _{T}\varphi T=-sgn(q)\sqrt{1-\sum\limits_{\alpha =1}^{s}\cos
^{2}\theta _{\alpha }}\left( -\kappa _{1}T+\kappa _{2}v_{3}\right)
\label{3.11}
\end{equation}%
In view of (\ref{3.6}), (\ref{3.10}) and (\ref{3.11}), it is easy to see that%
\begin{equation}
q\left[ \sum\limits_{\alpha =1}^{s}\cos \theta _{\alpha }\xi _{\alpha
}-\left( \sum\limits_{\alpha =1}^{s}\cos ^{2}\theta _{\alpha }\right) T%
\right] =sgn(q)\sqrt{1-\sum\limits_{\alpha =1}^{s}\cos ^{2}\theta _{\alpha }}%
\kappa _{2}v_{3}  \label{3.12}
\end{equation}%
\begin{equation}
\kappa _{2}=\left\vert q\right\vert \sqrt{\sum\limits_{\alpha =1}^{s}\cos
^{2}\theta _{\alpha }}  \label{3.13}
\end{equation}%
If we write (\ref{3.13}) in (\ref{3.12}), we have%
\begin{equation}
\sum\limits_{\alpha =1}^{s}\cos \theta _{\alpha }\xi _{\alpha }=\left(
\sum\limits_{\alpha =1}^{s}\cos ^{2}\theta _{\alpha }\right) T+\sqrt{%
\sum\limits_{\alpha =1}^{s}\cos ^{2}\theta _{\alpha }}\sqrt{%
1-\sum\limits_{\alpha =1}^{s}\cos ^{2}\theta _{\alpha }}v_{3}  \label{3.14}
\end{equation}%
If we differentiate (\ref{3.14}), we find $\kappa _{3}=0$. From equations (%
\ref{3.7}) and (\ref{3.14}), we can write%
\begin{equation}
v_{2}=\frac{-sgn(q)}{\sqrt{1-\sum\limits_{\alpha =1}^{s}\cos ^{2}\theta
_{\alpha }}}\varphi T  \label{v2}
\end{equation}%
\begin{equation}
v_{3}=\frac{1}{\sqrt{\sum\limits_{\alpha =1}^{s}\cos ^{2}\theta _{\alpha }}%
\sqrt{1-\sum\limits_{\alpha =1}^{s}\cos ^{2}\theta _{\alpha }}}\left(
\sum\limits_{\alpha =1}^{s}\cos \theta _{\alpha }\xi _{\alpha }-\left(
\sum\limits_{\alpha =1}^{s}\cos ^{2}\theta _{\alpha }\right) T\right)
\label{v3}
\end{equation}%
Finally, if $\kappa _{1}=0$, after some calculations, by (\ref{2.1}) and (%
\ref{3.7}), we obtain $T=\pm \sum\limits_{\alpha =1}^{s}\cos \theta _{\alpha
}\xi _{\alpha }$, where $\sum\limits_{\alpha =1}^{s}\cos ^{2}\theta _{\alpha
}=1$. So, we can give the following theorem:

\begin{theorem}
\label{theorem1}Let $\gamma :I\rightarrow M=(M,\varphi ,\xi _{\alpha },\eta
^{\alpha },g)$ be a unit-speed curve in a $C$-manifold. Then $\gamma $ is a
normal magnetic curve for $F_{q}$ $(q\neq 0)$ in $M$ if and only if

i) $\gamma $ is a geodesic $\theta _{\alpha }-$slant curve as an integral
curve of $\pm \sum\limits_{\alpha =1}^{s}\cos \theta _{\alpha }\xi _{\alpha
} $, where $\sum\limits_{\alpha =1}^{s}\cos ^{2}\theta _{\alpha }=1$; or

ii) $\gamma $ is a Legendre circle with $\kappa _{1}=\left\vert q\right\vert 
$ having the Frenet frame field%
\begin{equation*}
\left\{ T,-sgn(q)\varphi T\right\} ;
\end{equation*}%
or

iii) $\gamma $ is a non-Legendre $\theta _{\alpha }-$slant helix with 
\begin{equation*}
\kappa _{1}=\left\vert q\right\vert \sqrt{1-\sum\limits_{\alpha =1}^{s}\cos
^{2}\theta _{\alpha }},
\end{equation*}%
\begin{equation*}
\kappa _{2}=\left\vert q\right\vert \sqrt{\sum\limits_{\alpha =1}^{s}\cos
^{2}\theta _{\alpha }},
\end{equation*}%
having the Frenet frame field%
\begin{equation*}
\left\{ T,v_{2},v_{3}\right\} ,
\end{equation*}%
where $\sum\limits_{\alpha =1}^{s}\cos ^{2}\theta _{\alpha }<1$, $v_{2}$ and 
$v_{3}$ are given in equations (\ref{v2}) and (\ref{v3}), respectively.
\end{theorem}

\begin{corollary}
\label{corollary1}If $\gamma $ is a unit-speed slant curve in $M$, then it
is a normal magnetic curve if anf only if

i) it is a geodesic as an integral curve of $\frac{\pm 1}{\sqrt{s}}%
\sum\limits_{\alpha =1}^{s}\xi _{\alpha }$; or

ii) $\gamma $ is a Legendre circle with $\kappa _{1}=\left\vert q\right\vert 
$ having the Frenet frame field%
\begin{equation*}
\left\{ T,-sgn(q)\varphi T\right\} ;
\end{equation*}%
or

iii) $\gamma $ is a non-Legendre slant helix with $\kappa _{1}=\left\vert
q\right\vert \sqrt{1-s\cos ^{2}\theta },$ $\kappa _{2}=\left\vert
q\right\vert \sqrt{s}\varepsilon \cos \theta ,$ having the Frenet frame field%
\begin{equation*}
\left\{ T,\frac{-sgn(q)}{\sqrt{1-s\cos ^{2}\theta }}\varphi T,\frac{%
\varepsilon }{\sqrt{s}\sqrt{1-s\cos ^{2}\theta }}\left( \sum\limits_{\alpha
=1}^{s}\xi _{\alpha }-s\cos \theta T\right) \right\} ,
\end{equation*}%
where $\theta \neq \frac{\pi }{2}$ is the contact angle satisfying $%
\left\vert \cos \theta \right\vert <\frac{1}{\sqrt{s}}$ and $\varepsilon
=sgn\left( \cos \theta \right) .$
\end{corollary}

\begin{proof}
Since $\theta _{\alpha }=\theta $ for all $\alpha =1,2,...,s,$ if we use 
\begin{equation*}
\sum\limits_{\alpha =1}^{s}\cos ^{2}\theta _{\alpha }=s\cos ^{2}\theta
\end{equation*}%
and 
\begin{equation*}
\sum\limits_{\alpha =1}^{s}\cos \theta _{\alpha }\xi _{\alpha }=\cos \theta
\sum\limits_{\alpha =1}^{s}\xi _{\alpha }
\end{equation*}%
in Theorem \ref{theorem1}, the proof is clear.
\end{proof}

\textbf{Remark. }If we take $s=1$, we have Proposition 1 in \cite%
{magneticcosymplectic}.

Let $M=(M,\varphi ,\xi _{\alpha },\eta ^{\alpha },g)$ be a $C$-manifold. A
Frenet curve of order $r=2$ is called a $\varphi $-curve in $M$ if $%
sp\left\{ T,v_{2},\xi _{1},...,\xi _{s}\right\} $ is a $\varphi -$invariant
space. A Frenet curve of order $r\geq 3$ is called a $\varphi $-curve if $%
sp\left\{ T,v_{2},...,v_{r}\right\} $ is $\varphi -$invariant. A $\varphi -$%
helix of order $r$ is a $\varphi -$curve with constant curvatures $\kappa
_{1},...,\kappa _{r-1}$. A $\varphi -$helix of order $3$ is shortly named a $%
\varphi -$helix.

\begin{proposition}
\label{prop}If $\gamma $ is a Legendre $\varphi -$helix in a $C$-manifold,
then it is a Legendre $\varphi -$circle.
\end{proposition}

\begin{proof}
Let $\gamma $ be a Legendre $\varphi -$helix. Then the contact angles $%
\theta _{\alpha }=\frac{\pi }{2}$ for all $\alpha =1,2,...,s$ and the Frenet
frame field $\left\{ T,v_{2},v_{3}\right\} $ is $\varphi -$invariant. Thus,
we can write%
\begin{equation}
g\left( \varphi T,v_{2}\right) =\cos \mu ,  \label{fitev2}
\end{equation}%
\begin{equation}
\varphi T=\cos \mu v_{2}\pm \sin \mu v_{3},  \label{fite}
\end{equation}%
for some function $\mu =\mu (t)$. If we differentiate equation (\ref{fitev2}%
), we find%
\begin{eqnarray}
-\mu ^{\prime }\sin \mu &=&\kappa _{2}g\left( \varphi T,v_{3}\right)
\label{starr} \\
&=&\pm \kappa _{2}\sin \mu .  \notag
\end{eqnarray}%
Firstly, let us assume that $\mu =0$, i.e. $\varphi T=v_{2}$. Hence, we have%
\begin{equation*}
\nabla _{T}\varphi T=-\kappa _{1}T=-\kappa _{1}T+\kappa _{2}v_{3},
\end{equation*}%
which is equivalent to $\kappa _{2}=0$. Likewise, if $\mu =\pi $, we obtain $%
\kappa _{2}=0$. Finally, let us assume that $\mu \neq 0,\pi $. In this case,
since $\gamma $ is a helix, using (\ref{starr}), we have%
\begin{equation*}
\kappa _{1}=constant,
\end{equation*}%
\begin{equation*}
\kappa _{2}=\mp \mu ^{\prime }=constant.
\end{equation*}%
If we differentiate (\ref{fite}), we calculate%
\begin{equation*}
\kappa _{1}\varphi v_{2}=-\kappa _{1}\cos \mu T.
\end{equation*}%
If we apply $\varphi $ to both sides, we conclude $\varphi T=\pm v_{2}$,
which gives $\kappa _{2}=0$. This completes the proof.
\end{proof}

\textbf{Remark.} For $s=1$, we obtain Proposition 2 of \cite%
{magneticcosymplectic}. Likewise, the following theorem generalizes Theorem
1 of \cite{magneticcosymplectic} to $C$-manifolds:

\begin{theorem}
If $\gamma $ be $\varphi -$helix of order $r\leq 3$ in a $C$-manifold $%
M=(M,\varphi ,\xi _{\alpha },\eta ^{\alpha },g)$. Then, the following
statements are valid:

i) If $\cos \theta _{\alpha }$ ($\alpha =1,2,...,s$) are constants such that 
$\sum\limits_{\alpha =1}^{s}\cos ^{2}\theta _{\alpha }=1$, then $\gamma $ is
an integral curve of $\pm \sum\limits_{\alpha =1}^{s}\cos \theta _{\alpha
}\xi _{\alpha }$, hence it is a normal magnetic curve for arbitrary $q$.

ii) If $\cos \theta _{\alpha }=0$ for all $\alpha =1,2,...,s$, i.e. $\gamma $
is a Legendre $\varphi -$curve, then it is a magnetic circle generated by
the magnetic field $F_{\pm \kappa _{1}}$.

iii) If $\cos \theta _{\alpha }$ ($\alpha =1,2,...,s$) are constants such
that $\sum\limits_{\alpha =1}^{s}\cos ^{2}\theta _{\alpha }=\frac{\kappa
_{2}^{2}}{\kappa _{1}^{2}+\kappa _{2}^{2}}$, then $\gamma $ is a magnetic
curve for $F_{\pm \sqrt{\kappa _{1}^{2}+\kappa _{2}^{2}}}$.

iv) Except above cases, $\gamma $ cannot be a magnetic curve for any
magnetic field $F_{q}$.
\end{theorem}

\begin{proof}
In view of Theorem \ref{theorem1} and Proposition \ref{prop}, it is
straightforward to show that $\nabla _{T}T=-q\varphi T$ for valid $q$.
\end{proof}

\section{Magnetic Curves of $%
\mathbb{R}
^{2n+s}$ with its structures as a $C$-manifold}

In this section, we consider parametrizations of normal magnetic curves in $%
M=%
\mathbb{R}
^{2n+s}$ as a $C$-manifold. Let $\left\{
x_{1},...,x_{n},y_{1},...,y_{n},z_{1},...,z_{s}\right\} $ be the coordinate
functions and define 
\begin{equation*}
X_{i}=\frac{\partial }{\partial x_{i}},\text{ }Y_{i}=\frac{\partial }{%
\partial y_{i}},\text{ }\xi _{\alpha }=\frac{\partial }{\partial z_{\alpha }}%
\text{ },
\end{equation*}%
for $i=1,...,n$ and $\alpha =1,2,...,s$. $\left\{ X_{i},Y_{i},\xi _{\alpha
}\right\} $ is an orthonormal basis of $\chi \left( M\right) $ with respect
to the usual metric%
\begin{equation*}
g=\sum\limits_{i=1}^{n}\left[ \left( dx_{i}\right) ^{2}+\left( dy_{i}\right)
^{2}\right] +\sum\limits_{\alpha =1}^{s}\left( dz_{\alpha }\right) ^{2}.
\end{equation*}%
Let us define a $(1,1)$-type tensor field $\varphi $ as 
\begin{equation*}
\varphi X_{i}=-Y_{i},\text{ }\varphi Y_{i}=X_{i},\text{ }\varphi \xi
_{\alpha }=0.
\end{equation*}%
Finally, let $\eta ^{\alpha }=dz_{\alpha }$ for $\alpha =1,2,...,s$. It is
well-known that $\left( M,\varphi ,\xi _{\alpha },\eta ^{\alpha },g\right) $
is a $C$-manifold, since $d\eta ^{\alpha }=0$ and $d\Omega =0$, where $%
\Omega \left( X,Y\right) =g\left( X,\varphi Y\right) $ for all $X,Y\in \chi
(M)$ (see \cite{Blair-1970} and \cite{Blair-2010}).

Let us denote normal magnetic curve by 
\begin{equation*}
\gamma =\left( \gamma _{1},...,\gamma _{n},\gamma _{n+1},...,\gamma
_{2n},\gamma _{2n+1},...,\gamma _{2n+s}\right) .
\end{equation*}%
Then%
\begin{equation*}
T=\gamma ^{\prime }=\left( \gamma _{1}^{\prime },...,\gamma _{n}^{\prime
},\gamma _{n+1}^{\prime },...,\gamma _{2n}^{\prime },\gamma _{2n+1}^{\prime
},...,\gamma _{2n+s}^{\prime }\right) ,
\end{equation*}%
which gives us%
\begin{equation*}
\nabla _{T}T=\left( \gamma _{1}^{\prime \prime },...,\gamma _{n}^{\prime
\prime },\gamma _{n+1}^{\prime \prime },...,\gamma _{2n}^{\prime \prime
},\gamma _{2n+1}^{\prime \prime },...,\gamma _{2n+s}^{\prime \prime }\right)
,
\end{equation*}%
\begin{equation*}
\varphi T=\left( \gamma _{n+1}^{\prime },...,\gamma _{2n}^{\prime },-\gamma
_{1}^{\prime },...,-\gamma _{n}^{\prime },0,...,0\right) .
\end{equation*}%
Since%
\begin{equation*}
\nabla _{T}T=-q\varphi T,
\end{equation*}%
we have%
\begin{equation*}
\eta ^{\alpha }\left( T\right) =\gamma _{2n+\alpha }^{\prime }=\cos \theta
_{\alpha }=constant
\end{equation*}%
and%
\begin{equation*}
\gamma _{2n+\alpha }=\cos \theta _{\alpha }t+h_{\alpha }.
\end{equation*}%
We also get%
\begin{eqnarray*}
\gamma _{i}^{\prime \prime } &=&-q\gamma _{n+i}^{\prime }, \\
\gamma _{n+i}^{\prime \prime } &=&q\gamma _{i}^{\prime }
\end{eqnarray*}%
for $i=1,...,n$. As a result, we obtain 
\begin{equation*}
\gamma _{i}^{\prime }\gamma _{i}^{\prime \prime }+\gamma _{n+i}^{\prime
}\gamma _{n+i}^{\prime \prime }=0,
\end{equation*}%
i.e.%
\begin{equation*}
\left( \gamma _{i}^{\prime }\right) ^{2}+\left( \gamma _{n+i}^{\prime
}\right) ^{2}=c_{i}^{2}.
\end{equation*}%
If we consider differentiable functions $f_{i}:I\rightarrow 
\mathbb{R}
,$ we can write%
\begin{equation*}
\gamma _{i}^{\prime }=c_{i}\cos f_{i},
\end{equation*}%
\begin{equation*}
\gamma _{n+i}^{\prime }=c_{i}\sin f_{i}.
\end{equation*}%
After calculations, we find%
\begin{equation*}
f_{i}\left( t\right) =qt+d_{i}.
\end{equation*}%
Finally, we have%
\begin{equation*}
\gamma _{i}=\frac{c_{i}}{q}\sin \left( qt+d_{i}\right) +b_{i},
\end{equation*}%
\begin{equation*}
\gamma _{n+i}=\frac{-c_{i}}{q}\cos \left( qt+d_{i}\right) +b_{n+i}.
\end{equation*}%
Thus, we give the following theorem:

\begin{theorem}
The normal magnetic curves on $%
\mathbb{R}
^{2n+s}$ satisfying the Lorentz equation $\nabla _{T}T=-q\varphi T$ have the
parametric equations%
\begin{equation*}
\gamma _{i}=\frac{c_{i}}{q}\sin \left( qt+d_{i}\right) +b_{i},
\end{equation*}%
\begin{equation*}
\gamma _{n+i}=\frac{-c_{i}}{q}\cos \left( qt+d_{i}\right) +b_{n+i},
\end{equation*}%
\begin{equation*}
\gamma _{2n+\alpha }=\cos \theta _{\alpha }t+h_{\alpha },
\end{equation*}%
where $i=1,...,n,$ $\alpha =1,2,...,s,$ $b_{i},$ $h_{\alpha }$ are arbitrary
constants and $\theta _{\alpha }$ are the constant contact angles.
\end{theorem}


\begin{thebibliography}{9}
\bibitem{Adachi-1996} Adachi ,T.: \textit{Curvature bound and trajectories
for magnetic fields on a Hadamard surface}. Tsukuba J. Math. \textbf{20},
225--230, (1996).

\bibitem{BRCF} Barros M., Romero, A.,~Cabrerizo, J.~L.,~Fern\'{a}ndez, M.:~ 
\textit{The Gauss-Landau-Hall problem on Riemannian surfaces}. J. Math.
Phys. \textbf{46}, no. 11, 112905, 15 pp, (2005).

\bibitem{Blair-1970} Blair, D. E.: Geometry of manifolds with structural
group $U(n)\times O(s)$. J. Differential Geometry, \textbf{4}, 155-167,
(1970).

\bibitem{Blair-2010} Blair, D. E.: \textit{Riemannian Geometry of Contact
and Symplectic Manifolds}, 2nd ed., Progr. Math. 203, Birkhiiuser Boston,
Boston, MA, 2010.

\bibitem{Comtet-1987} Comtet, A.:~ \textit{On the Landau levels on the
hyperbolic plane}. Ann. Physics \textbf{173}, 185-209, (1987).

\bibitem{DIMN-2015} Dru\c{t}\u{a}-Romaniuc, S. L., Inoguchi, J., Munteanu,
M. I., Nistor, A. I.: \textit{Magnetic curves in Sasakian manifolds}.
Journal of Nonlinear Mathematical Physics, \textbf{22}, 428-447, (2015).

\bibitem{GO-2019} G\"{u}ven\c{c}, \c{S}., \"{O}zg\"{u}r, C.: \textit{On
slant magnetic curves in }$S$\textit{-manifolds}. J. Nonlinear Math. Phys. 
\textbf{26} (4), 536--554, (2019).

\bibitem{magneticcosymplectic} Dru\c{t}\u{a}-Romaniuc, S.-L., Inoguchi,
J.-I., Munteanu, M. I., Nistor, A. I.: \textit{Magnetic Curves in
Cosymplectic Manifolds}. Reports on Mathematical Physics.\ \textbf{78} (1),
33-48, (2016).
\end{thebibliography}
\end{document}